\documentclass{amsart}
\usepackage{amsmath}
\usepackage{amsfonts}
\usepackage{amsthm}
\usepackage{amssymb}
\usepackage[all]{xy}
\usepackage{color}
\usepackage{eucal}
\SelectTips{cm}{}

\DeclareMathOperator{\Ho}{Ho}

\DeclareMathOperator{\Hom}{Hom}

\DeclareMathOperator{\sset}{sSet}
\DeclareMathOperator{\Ch}{Ch}

\DeclareMathOperator{\Sp}{Sp}
\DeclareMathOperator{\map}{map}
\DeclareMathOperator{\Map}{Map}
\DeclareMathOperator{\holim}{holim}
\DeclareMathOperator{\hocolim}{hocolim}
\DeclareMathOperator{\colim}{colim}

\newcommand{\C}{\mathcal{C}}
\newcommand{\D}{\mathcal{D}}
\newcommand{\E}{\mathcal{E}}
\newcommand{\G}{\mathcal{G}}

\newcommand{\V}{\mathcal{V}}

\newcommand{\K}{\mathcal{K}}
\newcommand{\calS}{\mathcal{S}}

\theoremstyle{plain}
\newtheorem{theorem}{Theorem}[section]
\newtheorem{proposition}[theorem]{Proposition}
\newtheorem{corollary}[theorem]{Corollary}
\newtheorem{lemma}[theorem]{Lemma}

\newtheorem*{theorem*}{Theorem}

\theoremstyle{definition}
\newtheorem{definition}[theorem]{Definition}

\theoremstyle{remark}
\newtheorem{ex}[theorem]{Example}

\newtheorem{rmk}[theorem]{Remark}

\xyoption{all}

\begin{document}

\title{Bousfield localisations along Quillen bifunctors}
\author[J.J. Guti\'errez]{Javier J. Guti\'errez}
\address{Radboud Universiteit Nijmegen, Institute for
Mathematics, Astrophysics, and Particle Physics, Heyendaalseweg 135, 6525 AJ
Nijmegen, The Netherlands} 
\curraddr{Departament de Matem\`atiques i Inform\`atica,
Facultat de Matem\`atiques i Inform\`atica,
Universitat de Barcelona,
Gran Via de les Corts Catalanes 585,
08007 Barcelona, Spain}
\email{j.gutierrez@math.ru.nl}
\urladdr{http://www.math.ru.nl/~gutierrez}
\author[C. Roitzheim]{Constanze Roitzheim}
\address{University of Kent, School of Mathematics, Statistics and Actuarial
Science, Canterbury, Kent CT2 7NF, United Kingdom}
\email{C.Roitzheim@kent.ac.uk}
\urladdr{http://www.kent.ac.uk/smsas/personal/csrr}
\thanks{The first author was supported by the NWO (SPI 61-638) and the MEC-FEDER grants MTM2010-15831 and MTM2013-42178-P.
Both authors received support from the LMS Scheme~4 grant no. ~41360.}

\begin{abstract}
Consider a Quillen adjunction of two variables between combinatorial model
categories from $\C\times \D$ to $\E$, a set $\calS$ of morphisms in
$\C$ and a set $\K$ of objects in $\C$. We prove that there is a localised model structure $L_{\calS}\E$ on
$\E$,  where the local objects are the $\calS$-local objects in~$\E$
described via the right adjoint. Dually, we show that there is a colocalised model structure $C_{\K}\E$ on $\E$, where the colocal equivalences are the $\K$-colocal equivalences in~$\E$ described via the right adjoint. These localised and colocalised model structures generalise left and right Bousfield localisations of
simplicial model categories, Barnes and Roitzheim's familiar model
structures, and Barwick's enriched left and right Bousfield localisations.
\end{abstract}

\maketitle

\section*{Introduction}

Quillen adjunctions between spectra or spaces and other model categories are
a useful way to study homotopy structures. For example, one can gain insight
into a model category $\C$ by studying the canonical action of the homotopy category of
simplicial sets $\Ho(\sset)$ or of the stable homotopy category $\Ho(\Sp)$
(provided that $\C$ is a stable model category) on the homotopy category~$\Ho(\C)$.

In \cite{BR} it was studied how this set-up is compatible with homological
localisations of spectra, that is, left Bousfield localisation at
$E_*$-isomorphisms for a homology theory $E$. For a stable model category
$\C$, Barnes and Roitzheim constructed in~\cite{BR} a corresponding Bousfield
localisation $\C_E$ of $\C$ called \emph{stable $E$-familiarisation} with
appropriate universal properties. One of them implies that $\C_E$ is the
``closest'' model category to $\C$ such that every left Quillen functor from
the model category of symmetric spectra $\Sp$ to $\C_E$ factors over
$E$-local spectra $L_E\Sp$.

In this paper, we take this notion further by studying the compatibility of
Quillen adjunctions of two variables from $\C \times \D$ to $\E$ with
Bousfield localisations and colocalisations of $\C$ or~$\D$.  Given a Quillen adjunction of two
variables between combinatorial model structures
$$
-\otimes- \colon \C\times \D\longrightarrow \E,
$$
$$
\Hom_r(-,-)\colon \D^{op}\times \E\longrightarrow \C,
$$
$$
\Hom_l(-,-)\colon \C^{op}\times \E\longrightarrow \D,
$$
where $\mathcal{E}$ is left proper, and a set $\calS$ of morphisms in $\C$, we define an
$\mathcal{S}$-localised model structure $L_\calS \E$ on $\E$. The
cofibrations of $L_\calS \E$ are the same as the ones in $\E$, and the
fibrant objects are the objects $Z$ that are fibrant in $\E$ and such that
for every morphism $f\colon A\to B$ in $\mathcal{S}$ the induced map
$$
f^*\colon\Hom_l(B, Z)\longrightarrow \Hom_l(A, Z)
$$
is a weak equivalence in $\D$. We show that the fibrant objects of $L_\calS
\E$ can be equivalently characterised in terms of a set of homotopy
generators $\mathcal{G}_\D$ of $\D$. Namely, $Z$ is fibrant in $L_\calS \E$
if it is fibrant in $\E$ and $\Hom_r(G, Z)$ is $\mathcal{S}$-local in $\C$
for every $G$ in $\mathcal{G}_\D$.

In fact, if $I_\D$ denotes the set of generating cofibrations of $\D$ and
$\G_\D$ a set of homotopy generators of $\D$, then the model structure
$L_{\calS}\E$ can be obtained as the left Bousfield localisation of $\E$ with
respect to the set of morphisms $\calS\square I_\D$, where $\square$ denotes
the pushout-product, or equivalently as the left Bousfield localisation of
$\E$ with respect to $\calS\otimes \G_\D$. Moreover, we show that the left Quillen
bifunctor $\C\times \D\to \E$ induces a left Quillen bifunctor between the
localised model structures $L_\calS \C\times \D\to L_\calS \E$, and that the
model structure $L_{\calS}\E$ is the ``closest'' model structure to $\E$ with
this property.

Dually, if we start with a set of objects $\mathcal{K}$ in $\mathcal{C}$, we can obtain similar statements for
the corresponding $\mathcal{K}$-colocalised model structure $C_{\mathcal{K}}\mathcal{E}$.

The $\mathcal S$-localised model structure and the $\mathcal{K}$-colocalised model structure are 
useful as they now generalise two known constructions: the enriched left and right Bousfield
localisations of enriched model categories~\cite{Bar10} and the
$E$\nobreakdash-fami\-lia\-ri\-sation of spectral model
categories~\cite[Section 5]{BR}.

One application of our results is the description of Postnikov sections of arbitrary left proper
combinatorial model categories.  For the category $\sset$ of simplicial sets,
the model structure $P_k\sset$ for $k$th Postnikov sections is obtained via
localising $\sset$ with respect to the map $f_k:{S}^{k+1} \to {D}^{k+2}$.
Using our localisation  construction and combining it with the theory of
framings \cite{Hov99} we can now consider Postnikov sections $P_k\C$ in model
categories $\C$ that are not necessarily simplicial, for instance, when
$\C=\Ch_b(R)$ is the category of non-negatively graded chain complexes of $R$\nobreakdash-modules
endowed with the standard projective model structure.

The $\mathcal{S}$-localised and $\mathcal{K}$-colocalised model structures play also an important role in the study of towers and homotopy pullbacks of model categories, such as Postnikov towers, Bousfield arithmetic squares, and homotopy fibers~\cite{GR16}.

\bigskip
The paper is organised as follows. In Section 1, we recall some terminology
and basic results on locally presentable categories and combinatorial model
categories.  In Section 2, we discuss how Quillen adjunctions of two variables are compatible
with left and right Bousfield localisations. Given a Quillen adjunction of
two variables $\C \times \D \to \E$ we describe Bousfield localisations of
$\E$ based on localisations of $\C$ or $\D$ and their universal properties.
As particular examples, we recover enriched localisations~\cite{Bar10},
enriched colocalisations and $E$-familiarisations~\cite{BR14, BR}. Finally,
in Section 3 we study the special case of Postnikov $k$-types in
combinatorial model categories.
\bigskip

\noindent{\bf Acknowledgements.} The first author would like to thank Dimitri
Ara for many useful conversations on some of the topics of this paper. The
second author would like to thank David Barnes for motivating discussions and
the Radboud Universiteit Nijmegen for their hospitality.

\section{Review of combinatorial model categories}
In this section, we recall some terminology on locally presentable categories
and combinatorial model categories. The essentials of the theory of locally
presentable categories can be found in \cite{AR}, \cite{GU} or \cite{MP}.
Foundations on the theory of combinatorial model categories may be found in
\cite{Beke}, \cite{Dug01} and \cite{Lu}. As in~\cite{Hir03} and~\cite{Hov99}
we will assume that all our model categories are equipped with functorial
factorisations.

\subsection{Locally presentable categories}
Let $\lambda$ be a regular cardinal. A small category $\mathcal{I}$ is called
\emph{$\lambda$\nobreakdash-filtered} if it is nonempty and satisfies the
following two conditions:
\begin{itemize}
\item[{\rm (i)}] Given any set of objects $\{a_i\mid i\in I\}$ in
    $\mathcal{I}$, where $|I|<\lambda$, there is an object $a$ and a
    morphism $a_i\to a$ for each $i\in I$.
\item[{\rm (ii)}] Given any set of parallel morphisms $\{\alpha_j\colon
    a\to a' \mid j\in J\}$ in $\mathcal{I}$ between two fixed objects,
    where $|J|<\lambda$, there is a morphism $\gamma\colon a'\to a''$
    such that $\gamma\circ\alpha_j=\gamma\circ\alpha_{j'}$ for all $j, j'\in
    J$.
\end{itemize}
An object $X$ of a category $\C$ is called
\emph{$\lambda$\nobreakdash-presentable} if the functor $\C(X,-)$ from $\C$
to sets preserves $\lambda$\nobreakdash-filtered colimits.

A~cocomplete category $\C$ is \emph{locally
$\lambda$\nobreakdash-pres\-ent\-able} if there is a set of
$\lambda$-presentable objects $\mathcal{A}$ such that every object of $\C$ is
a $\lambda$-filtered colimit of objects from $\mathcal{A}$. In fact, if $\C$
is locally $\lambda$-presentable, then the collection of all $\lambda$-presentable objects has a set of
representatives with respect to isomorphism, and we will denote by $\mathcal{C}_{\lambda}$ the full subcategory determined
by any such set. The overcategory $\mathcal{C}_{\lambda}\downarrow X$ is
$\lambda$-filtered and if we denote by $\colim(\mathcal{C}_{\lambda}\downarrow X)$ the colimit of the diagram 
$\mathcal{C}_{\lambda}\downarrow X\to \mathcal{C}_{\lambda}\hookrightarrow \mathcal{C}$, where the first functor is the projection, then the canonical map
$$
\colim (\C_{\lambda}\downarrow X)\longrightarrow X
$$
is an isomorphism. A category is \emph{locally presentable} if it is locally $\lambda$-presentable for some regular cardinal $\lambda$.

Every locally $\lambda$-presentable category is equivalent to a full,
reflective subcategory closed under $\lambda$-filtered colimits of the
category of presheaves on some small category; see \cite[Proposition
1.46]{AR}.

\subsection{Combinatorial model categories}
A model category $\C$ is \emph{cofibrantly generated} if there exists a set $I_{\C}$
of \emph{generating cofibrations} and a set $J_{\C}$ of \emph{generating
trivial cofibrations} that one can use to perform the small object argument
(see \cite[Definition 11.1.2]{Hir03} or \cite[Definition 2.1.17]{Hov99} for a
precise definition).

A \emph{homotopy function complex} in a model category $\C$ is a functorial
choice of a fibrant simplicial set $\map_{\C}(X, Y)$, for every two objects
$X$ and $Y$ in $\C$, whose homotopy type is the same as the diagonal of the
bisimplicial set $\C(\widetilde{\mathbf X}, \widehat{\mathbf Y})$, where
$\widetilde{\mathbf X}$ is a cosimplicial resolution of $X$ and
$\widehat{\mathbf{Y}}$ is a simplicial resolution of $Y$; for more details,
see \cite[Chapter 17]{Hir03}. Functorial homotopy function complexes exist in
every model category and they are unique up to homotopy; see~\cite[Proposition 17.5.18 and Theorem 17.5.21]{Hir03}.

Let $\C$ be a model category with homotopy function complex $\map_\C(-,-)$
and let $i\colon A\to B$ and $p\colon X\to Y$ be two morphisms in $\C$. Then
the pair $(i,p)$ is a \emph{homotopy orthogonal pair} if the diagram
$$
\xymatrix{
\map_{\C}(B, X)\ar[r]\ar [d] & \map_{\C}(B, Y) \ar[d] \\
\map_{\C}(A, X)\ar[r] & \map_{\C}(A, Y)
}
$$
is a homotopy fiber square \cite[Definition 17.8.1]{Hir03}. The notion of homotopy orthogonal pair does not depend on the choice of homotopy function complexes, as shown in \cite[Proposition~17.8.2]{Hir03}. 
In particular, we have that 
the pair $(\emptyset\to W, p)$ is homotopy orthogonal if the induced map
$$
p_*\colon \map_\C(W, X)\longrightarrow \map_\C(W, Y)
$$
is a weak equivalence of simplicial sets.

Recall that a model category is \emph{left proper} if pushouts of weak
equivalences along cofibrations are weak equivalences, and \emph{right
proper} if pullbacks of weak equivalences along fibrations are weak
equivalences. A model category is \emph{proper} if it is left and right
proper.

In a cofibrantly generated model category the set of generating cofibrations
can be used to detect weak equivalences. A proof of the following result can
be found in \cite[Theorem 17.8.18]{Hir03}.

\begin{proposition}
Let $\C$ be a cofibrantly generated model category and let $I_{\C}$ be a set
of generating cofibrations. Assume that $\C$ is left proper or that the
domains of the elements of $I_{\C}$ are cofibrant. Then, a map $f$ in $\C$ is
a weak equivalence if and only if for every map $i$ in $I_{\C}$ the pair
$(i,f)$ is a homotopy orthogonal pair. $\hfill\qed$
\label{prop:gen_cof_detect}
\end{proposition}

A set of \emph{homotopy generators} for a model category $\C$ consists of a
small full subcategory $\mathcal{G}$ such that every object of $\C$ is weakly
equivalent to a filtered homotopy colimit of objects of $\mathcal{G}$. A set
of homotopy generators also detects weak equivalences.

\begin{proposition}\label{prop:generatorlift}
Let $\C$ be a model category with homotopy function complex $\map_\C(-,-)$
and a set of cofibrant homotopy generators $\mathcal{G}$. Then a map $f\colon
X\to Y$ in $\C$ is a weak equivalence if and only if for every $G$ in
$\mathcal{G}$ the pair $(j_{G}, f)$ is a homotopy orthogonal pair, where
$j_G$ denotes the morphism $\emptyset\to G$. \label{prop:hom_gen_detect}
\end{proposition}

\begin{proof}
Let $j_W$ denote  the map $\emptyset\to W$. By \cite[Theorem 17.7.7]{Hir03} a
map $f\colon X\to Y$ is a weak equivalence if and only if the pair $(j_W, f)$ is
a homotopy orthogonal pair for every object
$W$, that is, if and only if
the induced map
$$
f_*\colon\map_\C(W, X)\longrightarrow \map_{\C}(W, Y)
$$
is a weak equivalence. Let $\widehat{f}\colon \widehat{X}\to \widehat{Y}$ be a
fibrant approximation of $f$. By assumption every object $W$ is weakly equivalent to
a filtered homotopy colimit $\hocolim G_{\alpha}$ of objects of
$\mathcal{G}$, and hence \cite[Theorem 19.4.2(2)]{Hir03} and \cite[Theorem
19.4.4]{Hir03} imply that
$$
\map_{\C}(\hocolim G_{\alpha}, \widehat{X})\simeq\holim(\map_{\C}(G_\alpha, \widehat{X}))
$$
and that the map
$$
\holim(\map_{\C}(G_\alpha, \widehat{X}))\longrightarrow\holim(\map_{\C}(G_\alpha, \widehat{Y}))
$$
is a weak equivalence. The result now follows from the fact that homotopy
function complexes are homotopy invariant; see \cite[Theorem 17.7.7]{Hir03}.
\end{proof}

Let $\lambda$ be a regular cardinal. A~model category $\C$ is called
\emph{$\lambda$-combinatorial} if it is cofibrantly generated and the
underlying category is locally $\lambda$-presentable. A model category $\C$
is called \emph{combinatorial} if it is $\lambda$-combinatorial for some
regular cardinal~$\lambda$.

Every combinatorial model category is Quillen equivalent to a left Bousfield
localisation of a category of diagrams of simplicial sets equipped with the
projective model structure \cite[Theorem 1.1]{Dug01} and many model
categories of interest are combinatorial. Examples are pointed or unpointed
simplicial sets, pointed or unpointed motivic spaces~\cite{DRO, MV}, symmetric spectra over
simplicial sets \cite[\S\,3.4]{HSS} or over motivic spaces~\cite{Jar00}, module spectra
over a ring spectrum \cite[Theorem~4.1]{SS}, bounded or unbounded chain
complexes of modules over a ring \cite[\S\,2.3]{Hov99}, or any locally
presentable category equipped with the \emph{discrete} model structure, where
the weak equivalences are the isomorphisms and all morphisms are fibrations
and cofibrations.

Dugger also proved in \cite[Proposition 4.7]{Dug01} that every combinatorial
model category has a set of homotopy generators (in the sense that every object is a homotopy colimit of those) and that, moreover, they can
be chosen to be cofibrant. Recall that we denote by $\C\downarrow X$ the slice category
of $\C$ over an object $X$, and that by $\hocolim(\mathcal{C}_{\lambda}\downarrow X)$ we mean the homotopy colimit of the projection followed by inclusion functor $\mathcal{C}_{\lambda}\downarrow X\to \mathcal{C}_{\lambda}\hookrightarrow \mathcal{C}$.

\begin{proposition}[Dugger]
Let $\lambda$ be a regular cardinal and let $\C$ be a
$\lambda$-combina\-to\-rial model category. Let $\C_{\lambda}\subseteq \mathcal{C}$ denote the full
subcategory of the $\lambda$-presentable objects. Then every object $X$ is a
canonical filtered homotopy colimit of objects of~ $\C_{\lambda}$. More
precisely, the canonical map
\[
\hocolim(\C_{\lambda}\downarrow X)\longrightarrow X
 \]
is a weak equivalence. Moreover, there is a regular cardinal $\mu> \lambda$
such that the canonical map
\[
\hocolim(\C_{\mu}^{\rm cof}\downarrow X)\longrightarrow X
 \]
is a weak equivalence, where $\C_{\mu}^{\rm cof}$ denotes the full
subcategory of $\C_{\mu}$ consisting of the cofibrant objects. $\hfill\qed$
\end{proposition}

Given a combinatorial model category $\C$, we will denote by
$\mathcal{G}_{\C}$ a set of cofibrant homotopy generators, whose existence is guaranteed by the previous proposition.

\begin{corollary}
Let $\C$ be a combinatorial model category with a set of generating
cofibrations $I_{\C}$ and a set of cofibrant homotopy generators
$\mathcal{G}_\C$. Assume that $\C$ is left proper or that the domains of the
elements of $I_\C$ are cofibrant. Then, for every map $f$ in $\C$, the pair
$(i,f)$ is a homotopy orthogonal pair for all $i$ in $I_{\C}$ if and only for
every $G$ in $\mathcal{G}_\C$, the pair $(j_{G}, f)$ is a homotopy orthogonal
pair, where $j_G$ denotes the morphism $\emptyset\to G$.
\label{cor:equiv_gen}
\end{corollary}
\begin{proof}
This is a consequence of Proposition~\ref{prop:gen_cof_detect} and
Proposition~\ref{prop:hom_gen_detect}.
\end{proof}

\section{Left and right Bousfield localisations along Quillen bifunctors}
In this section we are going to discuss how Quillen adjunctions of two variables are compatible
with left and right Bousfield localisation.
\subsection{Quillen bifunctors}
Let $\C$, $\D$ and $\E$ be categories. An \emph{adjunction of two variables}
from $\C\times \D$ to $\E$ is given by functors
$$
-\otimes- \colon \C\times \D\longrightarrow \E,
$$
$$
\Hom_r(-,-)\colon \D^{op}\times \E\longrightarrow \C,
$$
$$
\Hom_l(-,-)\colon \C^{op}\times \E\longrightarrow \D,
$$
and natural isomorphisms
$$
\C(X, \Hom_r(Y, Z))\cong \E(X\otimes Y, Z)\cong\D(Y, \Hom_l(X, Z)).
$$

We will sometimes denote an adjunction of two variables from $\C\times \D$ to
$\E$ just by the left adjoint $\C\times \D\to \E$. The following analogous notion for
model categories appears in~\cite[Definition 4.2.1]{Hov99}.

\begin{definition}
Let $\C$, $\D$ and $\E$ be model categories. An adjunction of two variables
from $\C\times \D$ to $\E$ is a \emph{Quillen adjunction of two variables} if
for every cofibration $f\colon A\to B$ in $\C$ and every cofibration $g\colon
X\to Y$ in $\D$, the pushout-product$$ f\square g\colon B\otimes
X\coprod_{A\otimes X} A\otimes Y\longrightarrow B\otimes Y
$$
is a cofibration in $\E$ which is a trivial cofibration if either $f$ or $g$ is a trivial cofibration.
We will refer to the left adjoint $\otimes$ of a
Quillen adjunction of two variables as a \emph{left Quillen bifunctor}.
\end{definition}

Every Quillen adjunction of two variables induces a derived adjunction of two variables
on the corresponding homotopy categories.

\begin{rmk}
There are equivalent formulations of the previous condition satisfied by a
Quillen adjunction of two variables in terms of $\Hom_r^{\square}$ and
$\Hom_l^{\square}$, where $\Hom_r^{\square}$ and $\Hom_l^{\square}$ denote
the respective adjoints of the pushout-product;
see~\cite[Lemma~4.2.2]{Hov99}. Explicitly, an adjunction of two variables $\C\times \D\to\E$ is a Quillen
adjunction of two variables if and only if any of the following equivalent conditions hold:
\begin{itemize}
\item[{\rm (i)}] For every cofibration $f\colon A\to B$ in $\C$ and every fibration $h\colon U\to V$ in
$\E$, the map
$$
\Hom_l^{\square}(f,h)\colon \Hom_l(B,U)\longrightarrow \Hom_l(B, V)\times_{\Hom_l(A,V)}\Hom_l(A,U)
$$
is a fibration in $\D$ which is a trivial fibration if either $f$ is a trivial cofibration or $h$ is a trivial fibration.
\item[{\rm (ii)}] For every cofibration $g\colon X\to Y$ in $\D$ and every fibration $h\colon U\to V$ in
$\E$, the map
$$
\Hom_r^{\square}(g,h)\colon \Hom_r(Y,U)\longrightarrow \Hom_r(Y, V)\times_{\Hom_r(X,V)}\Hom_r(X,U)
$$
is a fibration in $\C$ which is a trivial fibration if either $g$ is a trivial cofibration or $h$ is a trivial fibration.
\end{itemize}
\end{rmk}

\begin{rmk}
If $(\otimes, \Hom_r, \Hom_l)$ is a Quillen adjunction of two variables from
$\C\times \D$ to $\E$ and $F_1\colon \C'\to \C$, $F_2\colon \D'\to \D$ and
$F_3\colon \E\to \E'$ are left Quillen functors (with right adjoints $G_1$,
$G_2$ and $G_3$, respectively), then
$$
(F_3(F_1(-)\otimes F_2(-)), G_1\Hom_r(F_2(-), G_3(-)), G_2\Hom_l(F_1(
-), G_3(-)))
$$
is a Quillen adjunction of two variables from $\C'\times \D'$ to $\E'$.
\end{rmk}

\begin{ex}
Let $\sset$ denote the category of simplicial sets with the Kan--Quillen
model structure. A \emph{simplicial model structure} on a model category $\C$
is the same as a left Quillen bifunctor $\C\times \sset\to \C$. A
\emph{topological model structure} can be defined similarly, by replacing
simplicial sets with the category of compactly generated Hausdorff spaces
equipped with the Quillen model structure.

Let $(\E, \otimes, I, \Hom_{\E})$ be a closed symmetric monoidal category. A
model structure on $\E$ is called a \emph{monoidal model structure} if
$-\otimes -\colon \E\times \E\to \E$ is a left Quillen bifunctor and the unit $I$
is cofibrant.

Let $\E$ be a monoidal model category. An \emph{$\E$-model category} is a
category $\C$ enriched, tensored and cotensored over $\E$ together with a
model structure such that the tensor, enrichment and cotensor define a
Quillen adjunction of two variables.
\end{ex}
The following two lemmas are an immediate consequence of the bifunctor
adjunctions and we state them without proof. We will use the terminology $f
\pitchfork g$ to indicate that a morphism $f$ has the \emph{left lifting
property} with respect to $g$ (or that $g$ has the \emph{right lifting
property} with respect to $f$), that is, $f \pitchfork g$ if for every
commutative diagram of the form
$$
\xymatrix{
A\ar[d]_f \ar[r]^i & X\ar[d]^g\\
B\ar[r]_p\ar@{.>}[ur]^h & Y,
}
$$
there is a diagonal lifting $h$ such that $i=hf$ and $p=gh$.

\begin{lemma}\label{lemma1}
Let $(\otimes, \Hom_r, \Hom_l)$ be an adjunction of two variables from
$\C\times \D$ to $\E$ and let $f$, $g$ and $h$ be morphisms in $\C$, $\D$ and
$\E$, respectively. The following are equivalent:
\begin{itemize}
\item[{\rm (i)}] $(f \square g) \pitchfork h$.
\item[{\rm (ii)}] $f \pitchfork \Hom_r^\square(g,h)$.
\item[{\rm (iii)}] $g \pitchfork \Hom_l^\square(f,h)$. \qed
\end{itemize}
\end{lemma}

\begin{lemma}\label{lemma2}
Let $(\otimes, \Hom_r, \Hom_l)$ be an adjunction of two variables from
$\C\times \D$ to $\E$ between model categories.
\begin{itemize}
\item[{\rm (i)}] The following are equivalent:
\begin{itemize}
\item[{\rm (a)}] For every cofibration $f$ in $\C$ and a every cofibration
    $g$ in $\D$, the morphism $f \square g$ is a cofibration in $\E$.
\item[{\rm (b)}] For every a cofibration $g$ in $\D$ and for every trivial
    fibration $h$ in $\E$, the morphism $\Hom^\square_r(g,h)$ is a
    trivial fibration in $\C$.
\item[{\rm (c)}] For every cofibration $f$ in $\C$ and every trivial
    fibration $h$ in $\E$, the morphism $\Hom^\square_l(f,h)$ is a
    trivial fibration in $\D$.
\end{itemize}

\item[{\rm (ii)}] The following are equivalent:
\begin{itemize}
\item[{\rm (a)}] For every cofibration $f$ in $\C$ and every
    trivial cofibration $g$ in $\D$, the morphism $f \square g$ is a trivial
    cofibration in $\E$.
\item[{\rm (b)}] For every trivial cofibration $g$ in $\D$ and every fibration $h$
    in $\E$, the morphism $\Hom^\square_r(g,h)$ is a trivial fibration in
    $\C$.
\item[{\rm (c)}] For every cofibration $f$ in $\C$ and every
    fibration $h$ in $\E$, the morphism $\Hom^\square_l(f,h)$ is a
    fibration in $\D$.\qed
\end{itemize}
\end{itemize}
\end{lemma}

Note that in every Quillen adjunction of two variables, if $X$ is cofibrant in $\C$, then $X\otimes -$ is a left Quillen
functor with right adjoint $\Hom_l(X, -)$. Similarly, if $Y$ is cofibrant in
$\D$, then $-\otimes Y$ is a left Quillen functor with right adjoint
$\Hom_r(Y, -)$.

Just as in the case of Quillen functors (see \cite[Proposition 8.5.4]{Hir03})
we have the following proposition which will be useful to test whether an
adjunction of two variables is a Quillen adjunction of two variables. In order to prove it, we
will make use of the following key result, which appears as \cite[Lemma
7.14]{JT}.

\begin{lemma}\label{lemma3}
A cofibration in a model category is a trivial cofibration if and only if it
has the left lifting property with respect to every fibration between fibrant
objects. Dually, a fibration in a model category is a trivial fibration if
and only if it has the right lifting property with respect to every
cofibration between cofibrant objects. \qed
\end{lemma}

\begin{proposition}\label{prop:testobject}
Let $(\otimes, \Hom_r, \Hom_l)$ be an adjunction of two variables from $\C
\times \D$ to $\E$ between model categories. Suppose that if $g$ is a
cofibration (respectively trivial cofibration) in $\D$ and $h$ is a trivial
fibration  (respectively fibration) in $\E$, then $\Hom_r^\square(g,h)$ is a
trivial fibration in $\C$. Then the following are equivalent:
\begin{itemize}
\item[{\rm (i)}] $(\otimes, \Hom_r, \Hom_l)$ is a Quillen adjunction of
    two variables.
\item[{\rm (ii)}] Given a cofibration $g$ in $\D$ and a fibration between
    fibrant objects $\hat{h}$ in $\E$, the morphism
    $\Hom_r^\square(g,\hat{h})$ is a fibration in $\C$.
\item[{\rm (iii)}] Given a cofibration between cofibrant objects
    $\tilde{g}$ in $\D$ and a fibration $h$ in $\E$, the morphism
    $\Hom_r^\square(\tilde{g},h)$ is a fibration in $\C$.
\item[{\rm (iv)}]  Given a cofibration between cofibrant objects
    $\tilde{g}$ in $\D$ and a fibration between fibrant objects $\hat{h}$
    in $\E$, the morphism $\Hom_r^\square(\tilde{g},\hat{h})$ is a
    fibration in $\C$.
\end{itemize}
\end{proposition}

\begin{proof}
It is clear that (i) implies (ii), (iii) and (iv), that (ii) implies (iv) and
that (iii) implies (iv). It then suffices, for example, to prove that (ii)
implies (i) and that (iv) implies (ii).

In order to prove that (ii) implies (i), let $g$ be any cofibration in $\D$
and $h$ any fibration in $\E$. Then $\Hom_r^\square(g,h)$ is a fibration in
$\C$ if and only if for every trivial cofibration $j$ in $\C$, we have that
$j \pitchfork \Hom_r^\square(g,h)$. But by Lemma \ref{lemma1}, this is
equivalent to $(j \square g)\pitchfork h$, in other words, $j \square g$
being a trivial cofibration. Since by assumption condition (i)(b) of Lemma~\ref{lemma2} holds,
Lemma~\ref{lemma2}(i) implies that $j
\square g$ is a cofibration. Hence, by Lemma~\ref{lemma3} the previous
condition is equivalent to $(j \square g) \pitchfork \hat{h}$ for $\hat{h}$
being any fibration between fibrant objets in $\E$. Again, by Lemma
\ref{lemma1} this is equivalent to $j \pitchfork \Hom_r^\square(g,\hat{h})$
for $\hat{h}$ any fibration between fibrant objects. Since we are assuming
that $\Hom_r^\square(g,\hat{h})$ is a fibration in~$\C$, the last statement
is true, so we can conclude that $\Hom_r^\square(g,h)$ is a fibration for any
cofibration $g$ and fibration $h$ as required, which was the missing part for
$(\otimes, \Hom_r, \Hom_l)$ to be a Quillen adjunction of two variables.

That part (iv) implies (ii) is proved in a very similar way to the previous
point. Let $g$ be any cofibration in $\D$ and let $\hat{h}$ be a fibration
between fibrant objects in~$\E$. Then $\Hom_r^\square(g,\hat{h})$ is a
fibration in $\C$ if and only if $j \pitchfork \Hom_r^\square(g,\hat{h})$ for
every trivial cofibration $j$ in $\C$. By Lemma \ref{lemma1} this is
equivalent to $g \pitchfork \Hom_l^\square(j,\hat{h})$ for every trivial
cofibration $j$ in $\C$. By the assumption of the proposition and Lemma~\ref{lemma2}(ii) the morphism
$\Hom_l^\square(j,\hat{h})$ is a fibration, and therefore, by
Lemma~\ref{lemma3}, it is a trivial fibration if and only if $\tilde{g}
\pitchfork \Hom^\square_l(j,\hat{h})$ for every cofibration $\tilde{g}$
between cofibrant objects in $\D$. By adjunction, this is equivalent to
saying that $j \pitchfork \Hom_r^\square(\tilde{g},\hat{h})$ for every
trivial cofibration $j$ in~$\C$, every cofibration between cofibrant objects
$\tilde{g}$ in $\D$, and every fibration  between fibrant objects $\hat{h}$
in $\E$. But $\Hom_r^\square(\tilde{g},\hat{h})$ is a fibration, by
assumption, hence (iv) is equivalent to (ii), which is what we wanted to
prove.
\end{proof}

\begin{rmk}\label{rmk:interchange}
Note that if $(\otimes, \Hom_r,\Hom_l)$ is an adjunction of two variables from $\C\times \D$ to~$\E$ and
$\tau\colon \D\times\C\to \C\times \D$ is the functor that interchanges the
components, then $(\otimes\circ\tau, \Hom_l,\Hom_r)$ is an adjunction of two
variables from $\D\times \C$ to $\E$.
\end{rmk}

\subsection{Left and right Bousfield localisation}
We recall the notion of \emph{left Bousfield localisation} and \emph{right
Bousfield localisation} (also called \emph{Bousfield colocalisation}) for
model categories; see~\cite[Chapters 3--5]{Hir03}. 

Let $\C$ be a model category with homotopy function complex
$\map_\C(-,-)$ and let $\mathcal{S}$ be a class of morphisms of $\C$ and
$\mathcal{K}$ a class of objects in $\C$. We say that an object $Z$ in $\C$
is \emph{$\mathcal{S}$-local} if it is fibrant and for every morphism
$f\colon A\to B$ in $\mathcal{S}$ the induced map
$$
f^*\colon\map_\C(B, Z)\longrightarrow \map_\C(A, Z)
$$
is a weak equivalence of simplicial sets. We say that a map $g\colon
    X\to Y$ is an \emph{$\mathcal{S}$-local equivalence} if the induced map
$$
g^*\colon\map_\C(Y, Z)\longrightarrow \map_\C(X, Z)
$$
is a weak equivalence of simplicial sets for every $\mathcal{S}$-local
object $Z$.

We say that a map $h\colon X\to Y$
in $\C$ is a \emph{$\mathcal{K}$-colocal equivalence} if for every object $K$
in $\mathcal{K}$ the induced map
$$
h_*\colon\map_\C(K,X)\longrightarrow \map_\C(K, Y)
$$
is a weak equivalence of simplicial sets. We say that an object $W$ in $\C$ is \emph{$\K$-colocal}
if it is cofibrant and for every $\K$-colocal
equivalence $h$ the induced map
$$
h_*\colon\map_\C(W, X)\longrightarrow \map_\C(W, Y)
$$
is a weak equivalence of simplicial sets.

The \emph{left Bousfield localisation} of $\C$ with respect to $\mathcal{S}$ (if it
exists) is a new model structure $L_{\mathcal{S}}\C$ on $\C$ such that
\begin{itemize}
\item[{\rm (i)}] the cofibrations of $L_{\mathcal{S}}\C$ are the same as
    those of $\C$,
\item[{\rm (ii)}] the weak equivalences of $L_{\mathcal{S}}\C$ are the $\mathcal{S}$-local equivalences,
\item[{\rm (iii)}] the fibrant objects of $L_{\mathcal{S}}\C$ are the
    $\mathcal{S}$-local objects.
\end{itemize}
The $\calS$-local equivalences between $\calS$-local objects are weak
equivalences in $\C$.

The \emph{right Bousfield localisation} (or Bousfield colocalisation) of $\C$ with
respect to~$\mathcal{K}$ (if it exists) is a new model structure
$C_{\mathcal{K}}\C$ on $\C$ such that
\begin{itemize}
\item[{\rm (i)}] the fibrations of $C_{\K}\C$ are the same as those of
    $\C$,
\item[{\rm (ii)}] the weak equivalences of $C_{\K}\C$ are the
    $\K$-colocal equivalences,
\item[{\rm (iii)}] the cofibrant objects of $C_{\K}\C$ are the
    \emph{$\K$-colocal objects}.
\end{itemize}
The $\K$-colocal equivalences between $\K$-colocal objects are weak
equivalences in $\C$.

\begin{rmk}
In~\cite[Chapter 3]{Hir03}, both left and right Bousfield localisations are defined with respect to arbitrary classes of morphisms. Here we have chosen to define right Bousfield localisations with respect to objects, since the main existence result works under the assumption that we localise at the class of $\mathcal{K}$-colocal equivalences for a set of objects $\mathcal{K}$.
\end{rmk}

\begin{rmk}\label{rmk:S_cof}
Note that the definition of the $\calS$-local objects does not depend on the chosen
homotopy function complexes, since they are unique up to homotopy and also homotopy invariant. Therefore, we can
always replace the morphisms in $ \calS$ by weakly equivalent ones consisting
of cofibrations between cofibrant objects without changing the model
structure $L_{\calS}\C$. Hence, without loss of generality we will often
assume that when we localise with respect to a class of morphisms, these
morphisms are cofibrations between cofibrant objects.

Similarly, we can assume without loss of generality that when we colocalise
with respect to a class of objects, they are cofibrant.
\end{rmk}
There are two main classes of model categories where localisations with
respect to a \emph{set} of morphisms and colocalisations with respect to a
\emph{set} of objects are always known to exist. These are the cellular model categories
and the combinatorial model categories.
For both classes the assumption of left properness is needed for the existence of
left Bousfield localisations (see \cite[Theorem 4.1.1]{Hir03} and \cite[Theorem 4.7]{Bar10}) and
right properness in needed for the existence of right Bousfield localisation (see \cite[Theorem
5.1.1]{Hir03} and \cite[Proposition 5.13]{Bar10}). If $\C$ is left proper and
combinatorial (or cellular) and $\calS$ is a set of morphisms of $\C$, then
$L_{\calS}\C$ is also left proper and combinatorial (or cellular). If $\C$ is
right proper and combinatorial (or cellular) and $\K$ is a set of objects of
$\C$, then $C_\K\C$ is also right proper, but it is not cofibrantly generated
in general.

\begin{definition}
Let $\otimes\colon\C\times\D\to \E$ be a left Quillen bifunctor, where $\D$ is
cofibrantly generated with set of generating cofibrations $I_{\D}$ and set of
cofibrant homotopy generators $\mathcal{G}_\D$. Assume that $\E$ is proper
and combinatorial and let $\calS$ and $\K$ be sets of morphisms and objects
in $\C$, respectively.
\begin{itemize}
\item[{\rm (i)}] The \emph{$\calS$-local model structure} on~$\E$,
    denoted by $L_{\calS}\E$, is the left Bousfield localisation
    $L_{\calS\square I_\D}\E$ of~$\E$ with respect to ${\calS\square
    I_\D}$.
\item[{\rm (ii)}] The \emph{$\K$-colocal model structure} on~$\E$,
    denoted by $C_{\K}\E$ is the right Bousfield localisation
    $C_{\K\otimes \mathcal{G}_\D}\E$ of $\E$ with respect to $\K\otimes
    \mathcal{G}_\D$.
\end{itemize}
\end{definition}

\begin{rmk}
If $(\otimes, \Hom_r, \Hom_l)$ is a Quillen adjunction of two variables from
$\C\times \D$ to $\E$, with $\C$ cofibrantly generated, $\E$ left proper and combinatorial, and $\calS$ is a set of morphisms in $\D$ (instead of
in $\C$), then we can also define an $\calS$-localised model structure on
$\E$ as $L_{I_{\C}\square \calS}\E$, where $I_\C$ is the set of generating
cofibrations of~$\C$. All the results from this section can be rephrased in
terms of a set of morphisms in $\D$, by suitably replacing $\Hom_l$ by
$\Hom_r$ and vice versa; see Remark~\ref{rmk:interchange}.\label{change_C_D}
\end{rmk}

\begin{theorem}\label{thm:localgenerators}
Let $(\otimes, \Hom_r, \Hom_l)$ be a Quillen adjunction of two variables from
$\C\times \D$ to $\E$. Let $\calS$ and $\K$ be classes of morphisms and
objects in $\C$, respectively. Assume that $\D$ is combinatorial with set of
generating cofibrations $I_\D$ and set of cofibrant homotopy
generators $\mathcal{G}_{\D}$ and that it is either left proper or the domains
of the elements of $I_\D$ are cofibrant.
\begin{itemize}
\item[{\rm (i)}] The following are equivalent for an object $Z$ of $\E$:
\begin{itemize}
\item[{\rm (a)}] $Z$ is $\calS\square I_\D$-local.
\item[{\rm (b)}] $Z$ is $\calS\otimes \mathcal{G}_\D$-local.
\item[{\rm (c)}] $Z$ is fibrant and $\Hom_r(G,Z)$ is $\calS$-local
    for every $G$ in $\mathcal{G}_\D$.
\item[{\rm (d)}] $Z$ is fibrant and for every $f\colon A\to B$ in
    $\calS$ the induced map
$$
f^*\colon\Hom_l(B, Z)\longrightarrow \Hom_l(A, Z)
$$
is a weak equivalence in $\D$.
\end{itemize}
\item[{\rm (ii)}] The following are equivalent for a morphism $h\colon
    X\to Y$ of $\E$:
\begin{itemize}
\item[{\rm (a)}] $h$ is a $\K\otimes\mathcal{G}_\D$-colocal
    equivalence.
\item[{\rm (b)}] For every $G$ in $\mathcal{G}_\D$ the induced map
$$
\hat{h}_*\colon \Hom_r(G,\hat X)\longrightarrow \Hom_r(G,\hat Y)
$$
is a $\K$-colocal equivalence, where $\hat{h}$ is a fibrant approximation of $h$.
\item[{\rm (c)}] For every $K$ in $\K$ the induced map
$$
\hat{h}_*\colon \Hom_l(K,\hat X)\longrightarrow \Hom_l(K,\hat Y)
$$
is a weak equivalence in $\D$, where $\hat{h}$ is a fibrant approximation of $h$.
\end{itemize}
\end{itemize}
\end{theorem}

\begin{proof}
We will prove part (i) first. Let $Z$ be any object of $\E$. Then $Z$ is $\calS\square I_\D$-local if and
only if it is fibrant and
$$
\map_{\E}(B\otimes Y, Z)\longrightarrow \map_{\E}(A\otimes Y\coprod_{A\otimes X} B\otimes X, Z)
$$
is a weak equivalence of simplicial sets for every map $A\to B$ in $\calS$
and every map $X\to Y$ in $I_\D$. By adjunction and the compatibility of
homotopy function complexes with Quillen pairs (see \cite[Proposition
17.4.16]{Hir03}), the previous condition is equivalent to the diagram
$$
\xymatrix{
\map_{\D}(Y, \Hom_l(B, Z))\ar[r]\ar[d] & \map_{\D}(Y, \Hom_l(A, Z))\ar[d] \\
\map_{\D}(X, \Hom_l(B, Z))\ar[r] & \map_{\D}(X, \Hom_l(A, Z))
}
$$
being a homotopy fiber square. This is the same as saying that for every
morphism $A\to B$ in $\calS$ and every morphism $X\to Y$ in~$I_{\D}$, the
pair given by the morphisms $X\to Y$ and $\Hom_l(B, Z)\to \Hom_l(A, Z)$ is a
homotopy orthogonal pair.

By Corollary~\ref{cor:equiv_gen} the previous condition amounts to saying
that the pair given by $\emptyset\to G$ and $\Hom_l(B, Z)\to \Hom_l(A, Z)$ is
a homotopy orthogonal pair for every $G$ in $\mathcal{G}_\D$, that is,
$$
\map_\D(G, \Hom_l(B,Z))\longrightarrow \map_\D(G, \Hom_l(A,Z))
$$
is a weak equivalence. Again by adjunction and the compatibility of homotopy
function complexes with Quillen adjunctions, this is equivalent to saying
that
$$
\map_\E(B\otimes G, Z)\longrightarrow \map_\E(A\otimes G, Z)
 $$
is a weak equivalence for every $G$ in $\mathcal{G}_{\D}$, and this is
precisely the condition of $Z$ being $\calS\otimes \mathcal{G}_\D$-local.
This proves that (a) and (b) are equivalent.

By adjunction (b) is equivalent to the fibrancy of $Z$ and the fact that
$$
\map_\C(B,\Hom_r(G,Z))\longrightarrow \map_\C(A, \Hom_r(G,Z))
$$
is a weak equivalence for every map $A\to B$ in $\calS$. Hence (b) and (c)
are equivalent.

Now, Proposition~\ref{prop:hom_gen_detect} shows that (b) is equivalent to
$\Hom_l(B, Z)\to \Hom_l(A, Z)$ being a weak equivalence in $\D$, which
concludes the proof of part (i).

To prove part (ii), first observe that a morphism $h\colon X\to Y$ is a
$\K\otimes\mathcal{G}_\D$-colocal equivalence if and only if $\hat h\colon
\hat X\to \hat Y$ is a $\K\otimes\mathcal{G}_\D$-colocal equivalence. By definition,
this means that
$$
\map_\E(K\otimes G, \hat X)\longrightarrow \map_\E(K\otimes G, \hat Y)
$$
is a weak equivalence of simplicial sets for every $K$ in $\K$ and every $G$
in $\mathcal{G}_\D$. As in the proof of part (i), by adjunction and the
compatibility of homotopy function complexes with Quillen adjunctions, this
is equivalent to  saying that
$$
\map_\C(K,\Hom_r(G,\hat X))\longrightarrow\map_\C(K,\Hom_r(G, \hat Y))
$$
is a weak equivalence for every $K$ in $\K$ and every $G$ in
$\mathcal{G}_\D$, or that
$$
\map_\D(G,\Hom_l(K,\hat X))\longrightarrow\map_\D(G,\Hom_l(K, \hat Y))
$$
is a weak equivalence for every $K$ in $\K$ and every $G$ in
$\mathcal{G}_\D$, which follows from Proposition~\ref{prop:generatorlift}.
\end{proof}

\begin{corollary}
\label{cor:equivalence} Let $\C$, $\D$ and $\E$ be left proper combinatorial
model categories and let $\otimes\colon\C\times \D\to \E$ be a left Quillen
bifunctor. Let $\calS$ be a set of morphisms in $\C$ and let $\mathcal{G}_\D$ be
a set of cofibrant homotopy generators of $\D$. Then
$L_{\calS}\E=L_{\calS\otimes \mathcal{G}_\D}\E$, where as before
$L_{\calS}\E$ means $L_{\calS\square I_{\D}}\E$.
\end{corollary}
\begin{proof}
The result follows immediately from Theorem~\ref{thm:localgenerators}.
\end{proof}

Recall that the left Bousfield localisation of a left proper combinatorial model category at a set of morphisms is cofibrantly generated.
\begin{proposition}\label{prop:lurie}
Let $\C$ be a left proper combinatorial model category and $\calS$ a set of morphisms in $\C$. If
$J_\calS$ is a set of generating trivial cofibrations of $L_\calS\C$, then
$L_\calS\C=L_{J_\calS}\C$.
\end{proposition}
\begin{proof}
The argument is the same as in \cite[Proposition A.3.7.4]{Lu}, where it is proved for left proper combinatorial simplicial
model categories. In our case we have to replace the simplicial enrichment by the homotopy function complex
$\map_\C(-,-)$.

Both model structures $L_{\calS}\C$ and $L_{J_\calS}\C$ have the same cofibrations, so it is enough to check that they have the same trivial cofibrations.  The elements in $J_\calS$ are trivial cofibrations in
$L_{J_\calS}\C$. Since the set $J_\calS$ determines the trivial cofibrations of $L_\calS \C$ (these are in fact the morphisms with the left lifting property with respect to the morphisms with the right lifting property
with respect to $J_\calS$) it follows that every trivial cofibration of $L_{\calS}\C$ is a trivial cofibration
of $L_{J_\calS}\C$.

Conversely, let $f\colon X\to Y$ be a trivial cofibration in $L_{J_\calS}\C$. It is in particular a cofibration
in $L_{\calS}\C$ so it suffices to see that it is an $\calS$-local equivalence. By assumption
$$
f^*\colon \map_\C(Y, Z)\longrightarrow\map_\C(X, Z)
$$
is a weak equivalence of simplicial sets for every $Z$ that is $J_\calS$-local. But every $\calS$-local
is $J_\calS$-local, since $J_\calS$ consists of trivial cofibrations of $L_\calS\C$, so $f^*$ is a weak equivalence
for every $Z$ that is $\calS$-local.
\end{proof}

\begin{proposition}
Let $\C$, $\D$ and $\E$ be left proper combinatorial model categories and let
$\otimes\colon\C\times \D\to \E$ be a left Quillen bifunctor. Let $\calS$ be a set
of morphisms in $\C$. Then $\otimes\colon L_{\calS}\C\times \D\rightarrow
L_{\calS}\E$ is a left Quillen
    bifunctor.
\label{prop:Quillenbif}
\end{proposition}
\begin{proof}
By \cite[Corollary 4.2.5]{Hov99} it is enough to prove that the
pushout-product axiom holds for the sets of generating cofibrations and
trivial cofibrations of $L_{\calS}\C$ and $\D$. As the cofibrations in
$L_{\calS}\C$ and $\C$ as well as the cofibrations in $L_{\calS}\E$ and $\E$
agree, it is sufficient to only consider the following case. Let $J_\calS$ be
a set of generating trivial cofibrations of $L_{\calS}\C$ and let $I_\D$ be a
set of generating cofibrations of $\D$. Since the cofibrations of
$L_{\calS}\C$ are the same as those in $\C$, it suffices to prove that if $i$
is in $J_\calS$ and $j$ is in $I_\D$, then $i\square j$ is a $\calS\square
I_{\D}$-local equivalence in $\E$. In fact, we will prove that the $J_\calS\square
I_{\D}$-local equivalences coincide with the $\calS\square I_{\D}$-local equivalences.

Let $\mathcal{G}_\D$ be a set of cofibrant homotopy generators of $\D$. By
Theorem~\ref{thm:localgenerators}(i), an object $Z$ of $\E$ is $\calS\square
I_\D$-local if and only if $\Hom_r(G,Z)$ is $\calS$-local for every $G$ in
$\mathcal{G}_D$. But by Proposition~\ref{prop:lurie}, $\calS$-local objects
coincide with $J_\calS$-local objects. Hence $\Hom_r(G,Z)$ is $J_\calS$-local
for every $G$ in $\mathcal{G}_\D$ and thus $Z$ is $J_\calS\square
I_\D$-local.
\end{proof}

\begin{proposition}\label{prop:char_S-equiv}
Let $\C$, $\D$ and $\E$ be model categories with sets of cofibrant homotopy
generators $\mathcal{G}_\C$, $\mathcal{G}_\D$ and $\mathcal{G}_\E$,
respectively. Suppose that $\D$ is left proper and combinatorial. Let
$(\otimes, \Hom_r, \Hom_l)$ be a Quillen adjunction of two variables from
$\C\times \D$ to $\E$ and let $\calS$ be a class of morphisms in $\C$. Let
$f\colon X\to Y$ be a map in $\E$ and let
    $\hat{f}\colon \hat{X}\to \hat{Y}$ be a fibrant
    approximation to~$f$ in~$L_{\mathcal{S}}\E$. If the induced map
$$
\hat{f}_*\colon\Hom_r(G, \hat{X})\longrightarrow \Hom_r(G, \hat{Y})
$$
is an $\calS$-local equivalence in $\C$ for every $G$ in $\mathcal{G}_\D$ and
$\mathcal{G}_\E\subset\mathcal{G}_\C\otimes \mathcal{G}_\D$, then $f$ is an
$\calS\otimes \mathcal{G}_\E$-local equivalence in $\E$.
\end{proposition}

\begin{proof}
By Theorem~\ref{thm:localgenerators}(i) the objects $\Hom_r(G, \hat{X})$ and
$\Hom_r(G, \hat{Y})$ are both $\calS$\nobreakdash-local. Thus $\hat{f}_*$ is
an $\calS$-local equivalence between $\calS$-local objects and hence a weak
equivalence in $\C$. This implies that
$$
\map_{\C}(W, \Hom_r(G,\hat{X}))\longrightarrow \map_{\C}(W,\Hom_r(G,\hat{Y}))
$$
is a weak equivalence of simplicial sets for every $W$ in $\mathcal{G}_\C$
and every $G$ in $\mathcal{G}_\D$. By adjunction and compatibility of
homotopy function complexes with Quillen functors this is equivalent to
$$
\map_{\E}(W\otimes G,\hat{X})\longrightarrow \map_{\E}(W\otimes G,\hat{Y})
$$
being a weak equivalence of simplicial sets for every $W$ in $\mathcal{G}_\C$
and every $G$ in $\mathcal{G}_\D$. Since by assumption
$\mathcal{G}_\E\subset\mathcal{G}_\C\otimes \mathcal{G}_\D$, this implies
that $\hat{f}$ is a weak equivalence in~$\E$. Now, by the 2-out-of-3 axiom
and the fact that weak equivalences in $\E$ are $\calS\otimes \mathcal{G}_\E$-local equivalences, it
follows that $f$ is an $\calS\otimes \mathcal{G}_\E$-local equivalence.
\end{proof}

The following definition is motivated by the notion of \emph{$E$-familiar model structure} described in~\cite[Section 4]{BarRoi11b}. 

\begin{definition}
Let $\C$ be a left proper combinatorial model category and let $\calS$ be a set of morphisms in $\C$. Let $\otimes\colon\C\times\D\to \E$ be a left Quillen bifunctor. We say that \emph{$\E$ is $\calS$-familiar} if
    $\otimes\colon L_{\calS}\C\times\D\to \E$ is a left Quillen bifunctor.
\end{definition}
\begin{rmk}
In particular, it follows from Proposition~\ref{prop:Quillenbif} that the
$\calS$-local model structure $L_{\calS}\E$ is $\calS$-familiar.
\end{rmk}

\begin{proposition}
Let $(\otimes, \Hom_r, \Hom_l)$ be a Quillen adjunction of two variables from
$\C\times\D$ to $\E$, where $\C$ is left proper and combinatorial, and let $\calS$ be a set of morphisms in $\C$. Then $\E$ is
$\calS$-familiar if and only if $\Hom_r(X, Y)$ is $\calS$-local for every $X$
cofibrant in $\D$ and $Y$ fibrant in $\E$.
\end{proposition}

\begin{proof}
The ``only if'' part follows from the fact that if $\E$ is $\calS$-familiar
and $X$ is cofibrant in $\D$, then the functor $\Hom_r(X,-)\colon \E\to
L_{\calS}\C$ is right Quillen. Hence, for every $Y$ fibrant in $\E$, we have
that $\Hom_r(X,Y)$ is fibrant in $L_{\calS}\C$, that is,
$\calS$-local.

Conversely, we want to show that if $\Hom_r(X,Y)$ is $\calS$-local for every
cofibrant $X$ and fibrant $Y$, then $L_{\calS}\C \times \D \to \E$ is also a
left Quillen bifunctor. Let $f$ be a cofibration (respectively, a trivial
cofibration) in $\D$ and let $g$ be a trivial fibration (respectively, a
fibration) in $\E$. Because $\C \times \D \to \E$ is assumed to be a left Quillen
bifunctor, the map $\Hom_r^{\square}(f,g)$ is a trivial fibration in $\C$ and hence
a trivial fibration in $L_{\calS}\C$ (since $\C$ and $L_{\calS}\C$ have the same cofibrations).
Therefore, by Proposition~\ref{prop:testobject} it suffices to prove that if
$f\colon A\to B$ is a cofibration between cofibrant objects in $\D$ and
$g\colon X\to Y$ is a fibration between fibrant objects in $\E$, then
$\Hom_r^{\square}(f,g)$ is a fibration in $L_{\calS}\C$.  Since it is already a fibration in $\C$, it is enough to check that the
source and target are $\mathcal{S}$-local. Consider the
pullback diagram
$$
\xymatrix{
\Hom_r(B, X)\ar[dr]^{\Hom_r^{\square}(f,g)}\ar@/^2pc/[rrd]^{g_*} \ar@/_2pc/[ddr]_{f^*} & &\\
& \Hom_r(B,Y)\times_{\Hom_r(A,Y)}\Hom_r(A, X)\ar[r]\ar[d] & \Hom_r(B, Y)\ar[d]^{f^*} \\
& \Hom_r(A, X) \ar[r]_{g_*} & \Hom_r(A, Y).
}
$$
The right vertical map $f^*$ is a fibration in $L_{\calS}\C$, since it is a
fibration in $\C$ between $\calS$\nobreakdash-local objects (see
\cite[Proposition 3.3.16]{Hir03}).  Since fibrations are closed under
pullbacks, the left vertical map is also a fibration in $L_{\calS}\C$. But
$\Hom_r(A,X)$ is  $\calS$\nobreakdash-local (that is, fibrant in
$L_{\calS}\C$) and therefore so is $$\Hom_r(B,Y)\times_{\Hom_r(A,Y)}\Hom_r(A,
X).$$

Hence, we have proved that $\Hom_r^{\square}(f,g)$ is a fibration in $\C$
between $\calS$-local objects. By \cite[Proposition 3.3.16]{Hir03} this means
that $\Hom_r^{\square}(f,g)$ is a fibration in~$L_{\calS}\C$.
\end{proof}

We have seen that for a left Quillen bifunctor $\otimes\colon\C \times \D \to \E$
and a set $\calS$ of morphisms in $\C$, the new model structure
$L_{\calS}\mathcal{E}$ on $\mathcal{E}$ gives rise to a left Quillen bifunctor
$$
\otimes\colon L_{\calS}\C \times \D \longrightarrow L_{\calS}\E.
$$
We can now state that this model structure $L_{\calS}\E$ is the ``closest''
model structure to $\E$ with this property in the following sense.

\begin{proposition}\label{prop:closest}
Let $\C$, $\D$ and $\E$ be left proper combinatorial model categories and let
$\otimes\colon\C\times \D\to \E$ be a left Quillen bifunctor. Let $F\colon \E \to
\E'$ be a left Quillen functor and $\calS$ a set of morphisms in $\C$.  If
$\E'$ is $\calS$\nobreakdash-fa\-mi\-liar with respect to the Quillen
bifunctor $F\circ \otimes\colon \C \times \D \to \E\to \E'$, then
$$
F\colon L_{\calS}\E \longrightarrow \E'
$$
is also a left Quillen functor, that is, $F$ factors over the
$\calS$-localisation of $\mathcal{E}$.
\end{proposition}

\begin{proof}
By Corollary~\ref{cor:equivalence} we have that $L_{\calS}\E=L_{\calS\otimes
\mathcal{G}_\D}\E$, where $\mathcal{G}_\D$ is a set of cofibrant homotopy
generators of $\D$. Thus, by \cite[Proposition 3.3.18]{Hir03} it is enough to
show that $F(f\otimes G)$ is a weak equivalence in $\E'$ for every $f$ in
$\calS$ and $G$ in $\mathcal{G}_\D$. But, by assumption, $F\circ \otimes
\colon L_{\calS}\times \D\to \E'$, is a left Quillen bifunctor. Hence $F(f\otimes
G)$ is a weak equivalence in $\E'$ since $f$ is a weak equivalence in
$L_{\calS}\C$ between cofibrant objects and $G$ is cofibrant in $\D$ (recall that, by Remark~\ref{rmk:S_cof}, we can assume without loss of
generality that the morphisms in $\calS$ are cofibrations between cofibrant objects).
\end{proof}

\subsection{Examples}
\subsubsection{Enriched localisations and colocalisations}
Let $\V$ be a monoidal model category and let $\C$ be a $\V$-enriched model
category. Then there is a Quillen adjunction of two variables $\C\times \V\to
\C$. If $\V$ is combinatorial, $\C$ is left proper combinatorial and $\calS$
is a set of morphisms in $\C$, then the $\calS$-localised model structure (see
Remark~\ref{change_C_D}) is the \emph{$\V$-enriched left Bousfield
localisation} of $\C$ with respect to $\calS$, as in \cite[Definition
4.42]{Bar10}. Similarly if $\K$ is a set of objects in $\C$, then the
$\K$-colocalised model structure of $\C$ along the left Quillen bifunctor is the
\emph{enriched right Bousfield localisation} of $\C$ with respect to $\K$.

If $\mathcal{V}=\sset$, the category of simplicial sets, then we recover left
and right Bousfield localisations of simplicial model categories.

\subsubsection{Familiarisations}
\label{familiarisation} Let $\C$ be a spectral model category, that is, a model category
which is compatibly enriched over the model category $\Sp$ of symmetric spectra. Then there is a
Quillen adjunction of two variables $\C\times \Sp\to \C$. Let $E$ be any spectrum and let
$\calS_E$ be the set of generating trivial cofibrations of the $E$\nobreakdash-local
model structure $L_E\Sp$; see \cite[Section1]{BarRoi11b} or \cite[Section 2]{BR14}. Then the $\calS_E$-localised model structure on
$\C$ is the \emph{$E$-familiarisation} of $\C$ in the sense of \cite[Section
5]{BR}.

If $\calS$ is a set of morphisms in $\Sp$, then we call the $\calS$-localised
model structure on~$\C$ the \emph{stable $\calS$-familiarisation}.

\section{Postnikov sections of model categories}
\label{k-types}

We are going to apply a construction closely related to our localisation construction to obtain Postnikov
sections in combinatorial model categories. We start by reviewing the
classical case of topological spaces and then explain how we can use our
construction to generalise this concept to arbitrary combinatorial model
categories which are not necessarily simplicial.

\subsection{The classical case: spaces}
We are going to recall some results for Postnikov towers and $k$-types in
simplicial sets. For details, see \cite[Section 1.5]{Hir03}. Note that
in~\cite{Hir03} this is formulated for topological spaces rather than
simplicial sets, but due to the compatibility of localisation with the
geometric realisation and total singular complex functors this will not be an
issue; see \cite[Section 1.6]{Hir03}.

Let $f_k\colon S^{k+1} \to D^{k+2}$ denote the boundary inclusion in $\sset$ from the
$(k+1)$-sphere to the $(k+2)$-disk. We form the left Bousfield localisation
of $\sset$ with respect to this map, obtaining the model structure
$L_{f_k}\sset$. This is called the \emph{model structure for $k$-types} of simplicial
sets. In fact, a simplicial set $X$ is $f_k$-local if and only if it is a Kan
complex and its homotopy groups vanish in degrees $k+1$ and higher, for every
choice of basepoint in $X$. The localisation map
$$
l_k\colon X \longrightarrow L_{f_k}X,
$$
which is defined as the fibrant replacement of $X$ in $L_{f_k}\sset$, is a
$\pi_i$-isomorphism for $i\le k$ and every choice of a basepoint in $X$.

\begin{rmk}
The model category $L_{f_k}\sset$ exists and is cellular (hence cofibrantly generated), since it
is a left Bousfield localisation of a left proper cellular model
category; see for example \cite[Theorem 4.1.1]{Hir03}.
\end{rmk}

\begin{proposition}\label{prop:weakequivalences}
If a map of fibrant simplicial sets $X\to Y$ is a $\pi_i$-isomorphism for $i
\le k$ and every choice of a basepoint in $X$, then it is an
$f_k$-local equivalence, that is, a weak equivalence in $L_{f_k}\sset$.
\end{proposition}
\begin{proof}
This is \cite[Propositions 1.5.2 and 1.5.4]{Hir03}.
\end{proof}

As a consequence of the above, we see that the localisation map $l_k$ of a
simplicial set $X$ to its $f_k$-localisation is nothing but the projection of
$X$ onto its $k$th Postnikov section $P_kX$; see~\cite[Theorem 1.5.3]{Hir03}. For details on Postnikov
sections, see for instance \cite[VI.3]{GJ99} or \cite[Section 4.3]{Hatcher}.

If $i\ge j$, then $P_j X$ is fibrant in $L_{f_i}\sset$, that is, $P_j X$ is $f_i$-local. Hence, there is a commutative triangle
\[
\xymatrix{X \ar[d]_{l_i} \ar[r]^-{l_j} & P_jX  \\
P_iX, \ar[ur]  & \\
}
\]
since, by definition, $l_i$ is a trivial cofibration in  $L_{f_i}\sset$.

\medskip
Furthermore, let $X \rightarrow Y$ be a weak equivalence in $L_{f_k}\sset$.
Consider the commutative square
\[
\xymatrix{ X \ar[r]\ar[d] & Y \ar[d] \\
P_k X \ar[r] & P_k Y.}
\]
We know that the vertical maps are $\pi_i$-isomorphisms for $i \le k$ by
definition. As the top horizontal and the two vertical maps are
$f_k$-local equivalences, then so is the map $P_k X \to P_k Y$. But of course $P_k
X$ and $P_k Y$ are $f_k$-local, so the bottom map is in fact a
$\pi_i$\nobreakdash-isomorphism for all $i$ and every choice of basepoint. Thus, any weak equivalence in
$L_{f_k} \sset$ is a $\pi_i$\nobreakdash-isomorphism for $i \le k$. Together with
Proposition~\ref{prop:weakequivalences} we can conclude that a map $X \to Y$ between fibrant simplicial sets is a
weak equivalence in $L_{f_k}\sset$ if and only if it is a $\pi_i$-isomorphism
for~$i \le k$ and every choice of basepoint.

\subsection{The general case}
\label{sect:k-types_model} Let $\C$ be now a simplicial, left proper,
combinatorial model category. Again, by $f_k$ we denote the map $S^{k+1} \to
D^{k+2}$ in simplicial sets, and denote $W_k=I_\C \square f_k$, where $I_\C$
denotes the set of generating cofibrations in $\C$ (see
Remark~\ref{change_C_D}). We then form the Bousfield localisation
$P_k\C=L_{W_k}\C$ which we will call the \emph{model structure for $k$-types}
in $\C$.

When $\C$ is a model category that is not necessarily simplicial (but still left proper and combinatorial), we can
still define the model structure for $k$-types in $\C$. In this case we use
the technique of \emph{framings}; see \cite[Section 5]{Hov99} or
\cite[Section 3]{BarRoi11b} for details. As explained in~\cite[Remark 5.2.9]{Hov99}, framings provide any model category
$\C$ with bifunctors \arraycolsep=0.5pt
\begin{eqnarray*}
-\otimes -&\colon & \C \times \sset \longrightarrow \C, \\
(-)^{(-)}&\colon& \sset^{op} \times \C \longrightarrow \C, \\
\map_l(-,-)&\colon& \C^{op} \times \C \longrightarrow \sset, \\
\map_r(-,-)&\colon& \C^{op} \times \C \longrightarrow \sset,
\end{eqnarray*}
and adjunctions
\[
\C(X\otimes K, Y)\cong \sset(K, \map_l(X,Y))\quad\mbox{and}\quad \C^{op}(Y^K, X)\cong \sset(K, \map_r(X,Y)).
\]
The homotopy function complex $\map_\C(-,-)$ agrees with the derived functors
$R\map_l(-,-)$ and $R\map_r(-,-)$. Moreover, if $X$ is a cofibrant object in
$\C$ and $Y$ is a fibrant object in $\C$, then
\[
\xymatrix@C-=0.5cm{
X\otimes -\colon\sset \ar@<3pt>[r] & \ar@<1pt>[l]\C\colon \map_l(X, -)
}\quad\mbox{and}\quad
\xymatrix@C-=0.5cm{
Y^{(-)}\colon\sset \ar@<3pt>[r] & \ar@<1pt>[l]\C^{op}\colon \map_r(-, Y).
}
\]
are Quillen pairs; see~\cite[Corollary 5.4.4]{Hov99}.

Note that a framing does not provide $\C$ with a simplicial model structure
though, as $\map_l$ and $\map_r$ only agree up to a zig-zag of weak
equivalences \cite[Proposition~5.4.7]{Hov99}. However, it does mean that
$\Ho(\C)$ is a closed $\Ho(\sset)$-module category. If $\C$ is already a
simplicial model category, the action from the simplicial structure agrees
with the $\Ho(\sset)$-action coming from framings. In our previous notation,
for a simplicial model category $\C$, the simplicial enrichment
$\Map(-,-)=\Hom_l(-,-)$ coincides with $\map_l(-,-)$ and $\map_r(-,-)$, and
the cotensor is $\Hom_r(-,-)$.

Thus, if our model category $\C$ is not simplicial we can define $W_k=I_\C
\square f_k$ just as before, where the pushout-product is constructed using
the functor $\otimes$ coming from the framing. (Note that if $\C$ is simplicial, the simplicial action on the homotopy category agrees with the action coming from the framing, so the rest of this section, particularly Definition \ref{def:ktypes}, does not depend on any choice of framing.)

\begin{rmk}
If $\C$ is a \emph{pointed} model category, then it is equipped with a
\emph{pointed framing} \cite[Section 5.7]{Hov99}, where the category of
simplicial sets is replaced by \emph{pointed} simplicial sets $\sset_*$.
\end{rmk}

\begin{definition}\label{def:ktypes}
Let $\C$ be a left proper combinatorial model category. We call
$P_k\C=L_{W_k}\C$ the \emph{model category of $k$-types in $\C$}. An object
of $\C$ is called a \emph{$k$-type} (or \emph{$k$-truncated}) if it is $W_k$-local, that is, fibrant in
$P_{k}\C$.
\end{definition}

Before we look further into the properties of this localisation, we need an
analogue of Theorem \ref{thm:localgenerators}(i) using framings. Note that
we are taking the class of maps~$\calS$ in $\sset$ (see
Remark~\ref{change_C_D}).

\begin{proposition}\label{prop:localgeneratorsframings}
Let $\C$ be a combinatorial, left proper model category with generating
cofibrations $I_\C$ and set of cofibrant homotopy generators
$\mathcal{G}_{\C}$. Furthermore, let $\calS$ be a class of maps in $\sset$.
Then the following are equivalent for an object $Z$ of $\C$:
\begin{itemize}
\item[{\rm (i)}] $Z$ is $I_\C\square \calS$-local
\item[{\rm (ii)}] $Z$ is $\mathcal{G}_\C\otimes \calS$-local
\item[{\rm (iii)}] $Z$ is fibrant and $\map_\C(G,Z)$ is $\calS$-local for
    every $G$ in $\mathcal{G}_\C$.
\item[{\rm (iv)}] $Z$ is fibrant and for every $g\colon X\to Y$ in
    $\calS$ the induced map
$$
g^*\colon Z^Y\longrightarrow Z^X
$$
is a weak equivalence in $\C$.
\end{itemize}
\end{proposition}

\begin{proof}
The proof follows exactly the same pattern as Theorem
\ref{thm:localgenerators}(i), so we are not spelling it out here. The
occurring functors $\otimes$, $\Hom_r$ and $\Hom_l$ have been replaced by the
functors $\otimes$, $(-)^{(-)}$, $\map_l$ and $\map_r$ coming from framings.
The only properties needed are that when $X$ is cofibrant and $Y$ is fibrant
in $\C$, the adjunctions $(X\otimes -, \map_l(X,-))$ and $(Y^{(-)}, \map_r(-,
Y))$ are Quillen pairs, and that $\map_l(X,Y)$ is weakly equivalent to
$\map_r(X, Y)$; see \cite[Proposition 5.4.7]{Hov99}. As the homotopy mapping
objects are also derived from framings, these are all compatible and the
necessary adjunctions hold just as before.
\end{proof}

\begin{proposition}\label{prop:framings}
Let $\C$ be a left proper combinatorial model category with set of cofibrant
homotopy generators $\mathcal{G}_\C$. A fibrant object $Z$ of $\C$ is a
$k$-type if and only if $\pi_i (\map_\C(X,Z))=0$ for all $X$ in $\C$, $i
>k$ and every choice of a basepoint, or equivalently, $\pi_i(\map_\C(G,Z))=0$ for all $G$ in $\mathcal{G}_\C$, $i>k$ and every
choice of a basepoint.
\end{proposition}

\begin{proof}
By Proposition~\ref{prop:localgeneratorsframings} we have that $Z$ is
$W_k$-local if and only if $Z$ is fibrant in~$\C$ and $\map_\C(G, Z)$ is a
$k$-type in $\sset$ for every $G$ in $\mathcal{G}_\C$. Since every object
in~$\C$ is weakly equivalent to a homotopy colimit of objects of
$\mathcal{G}_\C$ and those commute with homotopy function complexes, the
result follows.
\end{proof}

In combination with Proposition \ref{prop:localgeneratorsframings} we also
have the following.
\begin{corollary}
Let $\C$ be a left proper combinatorial model category with set of cofibrant
homotopy generators $\mathcal{G}_\C$, and let $f_k\colon S^{k+1}\to D^{k+2}$
in $\sset$. Then the model category of $k$-types $P_k\C$ coincides with
$L_{\mathcal{G}_\C \otimes f_k}\C$. \qed
\end{corollary}

\begin{rmk}
When $\C$ is a simplicial model category, then the model structure~$P_k\C$ agrees
with the model structure for $k$-types defined by Barwick in
\cite[Proposition~5.28]{Bar10}.
\end{rmk}

In the context of \emph{familiarisation} as defined by \cite{BR}, one would
define $P_{k}\C$ to be $L_{I_\C \square J_{f_k}}\C$ where $J_{f_k}$ denotes
the generating acyclic cofibrations of $L_{f_k}\sset$. However, those two
model structures agree since $L_{f_k}\sset=L_{J_{f_k}}\sset$ by
Proposition~\ref{prop:lurie}. The reason one works with the acyclic
cofibrations in \cite{BR} is to actually cut down the localised weak
equivalences of some $L_\calS\sset$ to a generating set if $\calS$ is not a
set. However, in our case we only localise simplicial sets at one morphism,
making this technicality unnecessary.

\begin{proposition}
Let $\C$ be a left proper combinatorial model category. The model category of
$k$-types $P_{k}\C$ has the following properties:
\begin{itemize}
\item[{\rm (i)}] Every Quillen adjunction $\sset \rightleftarrows \C$
    gives rise to a Quillen adjunction $L_{f_k}\sset \rightleftarrows
    P_{k}\C,$ and ${P_k}\C$ is the closest model structure to $\C$ with
    this property. This means that if $\C \rightleftarrows \D$ is a
    Quillen adjunction such that the composite $\sset \rightleftarrows
    \D$ factors over $L_{f_k}\sset$, then $P_k\C \rightleftarrows \D$ is
    also a Quillen adjunction.

\item[{\rm (ii)}] If $\C$ is a simplicial model category, then $P_{k}\C$
    is a $L_{f_k}\sset$-model category.

\item[{\rm (iii)}] For every $k\ge 0$ the model structures $P_{k}
    P_{{k+1}} \C$ and  $P_{k}\C$ coincide.
\end{itemize}
\end{proposition}

\begin{proof}

Let $F\colon\sset \rightleftarrows \C\colon U$ be a Quillen adjunction.
By~\cite[Proposition 3.3.18]{Hir03}, in order for this to be a Quillen
adjunction between $L_{f_k}\sset$ and $P_k\C$, we need to show that $F(f_k)$
is a weak equivalence in $P_k\C$.

By \cite[Chapter 5]{Hov99}, all Quillen
adjunctions such that the left adjoint is defined on the Quillen model structure on simplicial sets arise from framings, 
that is, for every left Quillen functor $F$ there is an object $A \in \C$ such that the left derived functors of $F$ and $A \otimes -$ agree. (Every adjunction between $\sset$ and $\C$ is of the
form $(A^\bullet \otimes -, \Hom(A^\bullet,-))$ for some cosimplicial object
$A^\bullet \in \C^\Delta$, and every Quillen adjunction is given by a framing
on $A^\bullet[0]=A$; see \cite[Proposition 3.1.5 and Section 5.2]{Hov99} and
\cite[Section~3]{BarRoi11b}.) So we have to show that $A \otimes f_k$ is a
weak equivalence in $P_k\C$. By
Proposition~\ref{prop:localgeneratorsframings}, all maps of the form $G
\otimes f_k$ are weak equivalences for all homotopy generators $G \in \mathcal{G}$.
But as every $A$ is a filtered homotopy colimit of such generators, and $- \otimes
f_k$ commutes with such homotopy colimits, $A \otimes f_k$ is a weak equivalence as well.

Now let $F'\colon \C \rightleftarrows \D\colon U'$ be another Quillen
adjunction such that $F'(F(f_k))$ is a weak equivalence in $\D$ for any left
Quillen functor $F$ as before. This means that $F'(A \otimes f_k)$ is a weak
equivalence in $\D$ for any $A \in \C$. So  in particular, $F'$ sends all
morphisms $G \otimes f_k$ to weak equivalences, where $G \in \G$. As
$P_k\C=L_{\G \otimes f_k}\C$, this means that $F'$ sends all the weak
equivalences in $P_k\C$ to weak equivalences in $\D$, which is what we wanted
to prove.

Part (ii) follows from Proposition~\ref{prop:Quillenbif}(ii), and part (iii)
follows from the fact that both model structures have the same cofibrations
and the same fibrant objects. This last point can be easily checked using the
characterisation of local objects given in
Proposition~\ref{prop:localgeneratorsframings}.
\end{proof}

Before we move on to the next result, let us note the following. The fact
that a model category is $\lambda$-presentable only depends on the underlying
category, not on its model structure. Also, the left Bousfield localisation
of a cofibrantly generated model category is again cofibrantly generated.
Thus, if a model category is combinatorial, so is any left Bousfield
localisation of it. Also, as Bousfield localisation does not change
cofibrations and preserves weak equivalences, if $\mathcal{G}_\C$ is a set of
homotopy generators for a combinatorial model category $\C$, then
$\mathcal{G}_\C$ will also be a set of homotopy generators for any left
Bousfield localisation of $\C$.

We can now characterise the weak equivalences of ${P_k}\C$.

\begin{proposition}
Let $\C$ be a left proper combinatorial model category and let $f\colon X \to Y$ be a morphism in $\C$. If its fibrant approximation
$\hat{f}\colon \hat{X} \to \hat{Y}$ in ${P_k}\C$ induces a
weak equivalence
$$
\hat{f}_*\colon \map_\C(G,\hat{X}) \longrightarrow \map_\C(G,\hat{Y})
$$
in $L_{f_k}\sset$ for all homotopy generators $G$ in $\mathcal{G}_\C$, then
the morphism $f$ is a weak equivalence in~${P_k}\C$.
\end{proposition}

\begin{proof}
We have that $\mathcal{G}_\C \subset \mathcal{G}_\C \otimes
\mathcal{G}_{\sset}$ as we can, without loss of generality, add the single
point to $\mathcal{G}_{\sset}$. Thus, the statement follows from
Proposition~\ref{prop:char_S-equiv}. Note that if $\C$ is not simplicial,
then we have to replace the mapping objects in that proof by the mapping
objects given by framings.
\end{proof}

\begin{corollary}
Let $\C$ be a left proper combinatorial model category and let $f\colon X \to Y$ be a morphism in $\C$. If its fibrant approximation
$\hat{f}\colon \hat{X} \to \hat{Y}$ in ${P_k}\C$ induces an
isomorphism of homotopy groups
$$
\pi_i(\hat{f}_*)\colon \pi_i(\map_\C(G,\hat X)) \longrightarrow
\pi_i(\map_\C(G,\hat Y))
$$
with respect to all basepoints for all $i\le k$ and homotopy generators $G$ in $\mathcal{G}_\C$, then $f$ is
a weak equivalence in~${P_k}\C$.  \qed
\end{corollary}

\subsection{Example: $\calS$-local simplicial sets}\label{sec:localsset}

Let us consider the example of left Bousfield localisations of pointed
simplicial sets, $\C=L_\calS\sset_*$. We can easily describe Postnikov
sections in this model category. By definition, $P_k L_\calS \sset_*=L_{W_k}
L_\calS\sset_*$ where $W_k= I_{L_\calS\sset_*} \square f_k$ and $f_k: S^{k+1}
\to D^{k+2}$. As the generating cofibrations $I_{L_\calS\sset_*}$ of
$L_\calS\sset_*$ are the same as the generating cofibrations of $\sset_*$ we have that
$I_{L_\calS\sset_*}\square f_k=I_{\sset_*}\square f_k$. Then we can conclude that $$P_k L_\calS\sset_*=L_{f_k}L_\calS\sset_*.$$ Thus, $X$ is
fibrant in $P_kL_\calS\sset_*$ if and only if it is a Kan complex,
$\calS$-local and $\pi_iX=\pi_iL_\calS X=0$ for $i > k$.

\subsection{Example: $k$-types in chain complexes} Let $\Ch_b(R)$ denote the category of \emph{non-negatively graded}
chain complexes of $R$-modules, where $R$ is a commutative ring with unit.
We are going to apply the results from the previous section this category. This is a particularly interesting example as it
concerns a model category that is not simplicial, although it is left proper and combinatorial. We are going to describe
the $k$-types in $\Ch_b(R)$ as well as describe some of the weak
equivalences. The results are just what one would expect and fit very neatly
with our general setup.

Consider the standard projective model structure on $\Ch_b(R)$; see~\cite[Section 7]{DS}. The
weak equivalences are given by quasi-isomorphisms, fibrations are morphisms
which are surjective in positive degrees, and cofibrations are monomorphisms
with projective cokernel in every degree. Consider the model category of
$k$-types of chain complexes, $P_k\Ch_b(R)$. According to Definition
\ref{def:ktypes}, this is the left Bousfield localisation with respect to the
set
$$
W_k=I_{\Ch_b(R)}\square \{ f_k\colon S^{k+1} \longrightarrow D^{k+2} \}.
$$
Now the generating cofibrations in the standard projective model structure
are given by the inclusions
$$
I_{\Ch_b(R)}= \{ \mathbb{S}^{n} \longrightarrow \mathbb{D}^{n+1} \mid n \ge 0\},
$$
where $\mathbb{S}^{n}$ denotes the chain complex which is $R$ in degree
$n$ and zero everywhere else, and $\mathbb{D}^{n+1}$ denotes the chain complex
with $R$ in degrees $n$ and $n+1$ with the identity differential between
them, and zero everywhere else. To avoid notational confusion with the sphere
and disk in spaces, we will use bold face for these.

Recall that the suspension functor $\Sigma$ in a pointed model category $\C$
can be defined using pointed framings; see \cite[Definition 6.1.1]{Hov99}. If
$X$ is a cofibrant object then $\Sigma X=X\otimes S^1$, that is, $\Sigma X$
is the pushout of the diagram
$$
\xymatrix{
X\otimes\partial\Delta[1] \ar[r]\ar[d] & X\otimes \Delta[1] \\
\ast &
}
$$
So let us look at the pushout-product 
\[
\xymatrix{ \mathbb{S}^n \otimes S^{k+1} \ar[r]\ar[d] & \mathbb{S}^n \otimes D^{k+2} \ar[d] \ar@/^/[ddr] & \\
\mathbb{D}^{n+1} \otimes S^{k+1} \ar@/_/[drr] \ar[r] & P \ar[dr] & \\
& & \mathbb{D}^{n+1} \otimes D^{k+2}.
}
\]

In the category $\Ch_b(R)$, the suspension is given by shifting. By \cite[Section 6.1]{Hov99}, $\mathbb{S}^i \otimes S^j = \mathbb{S}^{i+j}$. Furthermore, framings are compatible with fibre and cofibre sequences \cite[Section 6.2]{Hov99}, so the above diagram is the same as

\[
\xymatrix{ \mathbb{S}^{n+k+1} \ar[r]\ar[d] & \mathbb{D}^{n+k+2}  \ar[d] \ar@/^/[ddr] & \\
\mathbb{D}^{n+k+2}  \ar@/_/[drr] \ar[r] & P \ar[dr] & \\
& & \mathbb{D}^{n+k+3}.
}
\]
The pushout of two disks and a sphere is just another sphere,
hence we obtain
\[
W_k=\{ \mathbb{S}^{n+k+1} \longrightarrow \mathbb{D}^{n+k+2} \mid n \ge 0\},
\]
so $P_k\Ch_b(R)$ is just localising $\Ch_b(R)$ at the map $g_k\colon
\mathbb{S}^{k+1} \to \mathbb{D}^{k+2}$. Note that local equivalences are
closed under (positive) suspensions, and hence localising with respect to
$g_k$ is the same as localising with respect to $\{\Sigma^n g_k \mid n\ge
0\}=W_k$.

Recall that we denote by $\map_{\Ch_b(R)}(-,-)$ a homotopy function complex for the model category $\Ch_b(R)$.
\begin{proposition}\label{prop:ktypechaincx}
A fibrant chain complex $M$ in $\Ch_b(R)$ is a $k$-type if and only if
$H_i(M)=0$ for all $i > k$.
\end{proposition}

\begin{proof}
The chain complex $M$ is $g_k$-local if and only if
\[
\pi_i(\map_{\Ch_b(R)}(\mathbb{D}^{k+2},M)) \longrightarrow \pi_i(\map_{\Ch_b(R)}(\mathbb{S}^{k+1},M))
\]
is an isomorphism for all $i\ge 0$. By adjunction, this is equivalent to
\[
[ \mathbb{D}^{i+k+2}, M] \longrightarrow [ \mathbb{S}^{i+k+1}, M]
\]
being an isomorphism for all $i\ge 0$, where the square brackets denote
morphisms in the derived category $D_b(R)$. But as the chain complex
$\mathbb{D}^{i+k+2}$ is acyclic and the right hand side equals the homology
$H_{i+k+1}(M)$ of $M$, the above is equivalent to $H_i(M)=0$ for all $i >k$.
\end{proof}

We can now say something about the weak equivalences in $P_k\Ch_b(R)$. Recall
that if $M$ is a chain complex in $\Ch_b(R)$, we denote by $M[n]$ the
$n$-fold suspension of $M$.

\begin{proposition}\label{prop:chaincxequivalence}
Let $f\colon M \to N$ be a morphism of chain complexes such that $H_i(f)$ is
an isomorphism for $0\le i\le k+1$. Then $f$ is a weak equivalence in
$P_k\Ch_b(R)$.
\end{proposition}

\begin{proof}
This is very similar to \cite[Proposition 1.5.2]{Hir03}. Without loss of
generality, let $f\colon M \to N$ be a cofibration of chain complexes, that
is, a monomorphism with degreewise projective cokernel.
We know that $f$ is a weak equivalence in $P_k\Ch_b(R)$ if and only if
\[
\map_{\Ch_b(R)}(N,Z) \longrightarrow \map_{\Ch_b(R)}(M,Z)
\]
is an acyclic fibration in simplicial sets for all $g_k$-local $Z$; see
\cite[Section 1.3.1]{Hir03}. This is equivalent to having a lift in the
diagram
\[
\xymatrix{ \partial\Delta[n] \ar[r]\ar[d] & \map_{\Ch_b(R)}(N,Z) \ar[d] \\
\Delta[n] \ar[r] \ar@{.>}[ur] & \map_{\Ch_b(R)}(M,Z)
}
\]
for all $n\ge0$. By adjunction, this is equivalent to having a lift in the
diagram
\[
\xymatrix{M \otimes \Delta[n] \coprod\limits_{M \otimes \partial\Delta[n]} N \otimes \partial\Delta[n] \ar[r] \ar[d] & Z \ar[d] \\
N \otimes \Delta[n] \ar[r] \ar@{.>}[ur] & 0
}
\]
for all $n \ge 0$.

Note that for a chain complex $A$, the complex $A \otimes \Delta[n]$ is the $n$th suspension of the cone of $A$, while $A \otimes \partial \Delta[n]$ is $A[n] \oplus A[n+1]$ (direct sum of chain complexes). Thus, the top left corner of this square is the $n$th suspension of the mapping cone of~$f$, and the left vertical map is given by $f \oplus \rm{id}$.
As $f$ was assumed to be a homology isomorphism in degrees $0$ to $k+1$, this mapping cone is acyclic in degrees $0$ to $k+1$ and the left vertical map is a homology isomorphism in those degrees (as the cone of a chain complex is obviously acyclic).  

We know by Proposition \ref{prop:ktypechaincx} that $H_j(Z)=0$ for $j\ge
k+1$. This means that
we have a square in $\Ch_b(R)$ where the left vertical map is a cofibration
and the right vertical map a fibration. In order to have the desired lift,
one of those maps would have to be a homology isomorphism.

As the left vertical map is a homology isomorphism in degrees $0$ to $k+1$,
we can use the methods in \cite[Section 7.7, proof of MC4(i)]{DS} to construct a lift in
those degrees. Then we can use the same method as in \cite[Section 7.5, proof of MC4(ii)]{DS}
to inductively construct the lift from degrees $k+2$ onwards, which uses that
$H_j(N)=0$ for $j \ge k+1$.

So we have constructed a lift in the above square, which means that $f\colon
M \to N$ is a weak equivalence in $P_k\Ch_b(R)$.
\end{proof}

As a consequence of Proposition \ref{prop:ktypechaincx} and Proposition
\ref{prop:chaincxequivalence} we get the following.
\begin{corollary}
If $M$ is a chain complex in $\Ch_b(R)$, then the $W_k$-localisation is given
by the $k$-truncation $\tau_{\ge k} M$ of $X$, defined by
\[
(\tau_{\ge k} M)_n= \left\{
\begin{array}{ccc}
M_n &\mbox{ if } & n <k, \\
M_k/B_k &\mbox{ if } & n=k, \\
 0 & \mbox{ if } & n>k,
 \end{array}
 \right.
\]
where $B_k={\rm im}(d_k)$ denotes the group of $k$-boundaries. \qed
\end{corollary}

\end{document}